\documentclass{amsart}
\usepackage[utf8]{inputenc}
\usepackage[utf8]{inputenc}
\usepackage[T1]{fontenc}
\usepackage[all]{xy}
\usepackage{lipsum}
\usepackage{url}
\usepackage{tikz}
\usepackage{stackrel}
\usepackage{color}

\usepackage{amsmath,amsthm,amssymb}

\usepackage{xcolor}

\usepackage{graphicx}

\newtheorem{theorem}{Theorem}[section]
\newtheorem{lemma}[theorem]{Lemma}
\newtheorem{example}[theorem]{Example}
\newtheorem{proposition}[theorem]{Proposition}
\theoremstyle{definition}
\newtheorem{definition}[theorem]{Definition}

\newtheorem{remark}[theorem]{Remark}
\newtheorem{corollary}[theorem]{Corollary}

\numberwithin{equation}{section}

\begin{document}

\vspace{0.5in}

\renewcommand{\bf}{\bfseries}
\renewcommand{\sc}{\scshape}
\vspace{0.5in}

\title[Cross-sections of Milnor fibrations]%
{Cross-sections of Milnor fibrations and Motion planning \\ }

\author{Cesar A. Ipanaque Zapata}
\address{Departamento de Matem\'{a}tica,UNIVERSIDADE DE S\~{A}O PAULO
INSTITUTO DE CI\^{E}NCIAS MATEM\'{A}TICAS E DE COMPUTA\c{C}\~{A}O -
USP , Avenida Trabalhador S\~{a}o-carlense, 400 - Centro CEP:
13566-590 - S\~{a}o Carlos - SP, Brasil}
\curraddr{Departamento de Matem\'{a}ticas, CENTRO DE INVESTIGACI\'{O}N Y DE ESTUDIOS AVANZADOS DEL I. P. N.
Av. Instituto Polit\'{e}cnico Nacional n\'{u}mero 2508,
San Pedro Zacatenco, Mexico City 07000, M\'{e}xico}
\email{cesarzapata@usp.br}



\thanks {The author wishes to acknowledge support for this research from grant\#2018/23678-6 and grant2016/18714-8, S\~ao Paulo Research Foundation (FAPESP)}


\begin{abstract} Consider a fibration \[p:E\to S^{p-1}\] with fiber $F$. We have the following natural question: Under what conditions does this fibration admit a cross-section? Our purpose is to discuss this problem when the fibration $p$ is the Milnor fibration $f_{\mid}:f^{-1}(S^{p-1}_{\delta})\cap D^{n}_{\epsilon}\to S^{p-1}_{\delta}$ with Milnor fiber $F_f$ and the Milnor fibration of arrangements $Q:\mathcal{M}(\mathcal{A})\to \mathbb{C}^\ast$ with fiber $F=F(\mathcal{A})$. Furthermore, we use our results to study the tasking planning problem for the Milnor fibration as a work map and we give the tasking algorithms.  
\end{abstract}

\maketitle


\section{\bf Introduction}

Let $f:(\mathbb{R}^n,0)\longrightarrow (\mathbb{R}^p,0)$, $n> p\geq 2$, be an analytic map germ which satisfies the Milnor's conditions $(a)$ and $(b)$ (as an example, if $f$ has an isolated singularity at the origin), so one has the Milnor fibration (see \cite{dutertre2016}, Corollary 2.5),
\begin{equation}\label{milnor-fibration}
   f_{\mid}:f^{-1}(S^{p-1}_{\delta})\cap D^{n}_{\epsilon}\longrightarrow S^{p-1}_{\delta},  
\end{equation}
where $0<\delta\ll\epsilon\leq \epsilon_0,$ and $\epsilon_0$ is a Milnor's radius for $f$ at origin. 

\begin{figure}[!h]
    \centering
    \includegraphics[scale=0.6]{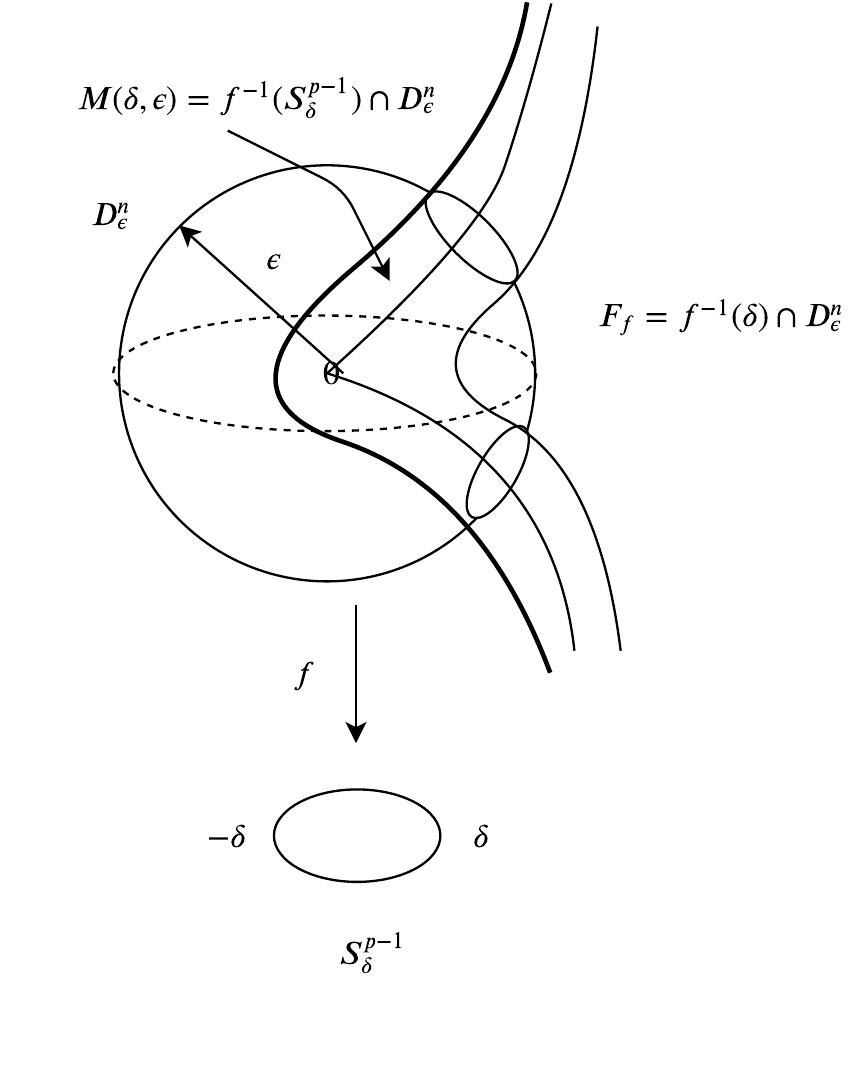}
    \caption{Milnor fibration}
    \label{fig:milnor-fibratio}
\end{figure}

Denote the \textit{Milnor tube} by $M(\delta,\epsilon)=f^{-1}(S^{p-1}_{\delta})\cap D^{n}_{\epsilon}$, and note that it has $dim(M(\delta,\epsilon))=n-1$.

Its fiber $f^{-1}(\delta)$ is called \textit{the Milnor fiber} denoted by $F_f$, which is a compact manifold with boundary. We have $dim~F_f=n-p$ and $\partial F_f$ is diffeomorphic to $K$, where $K:=f^{-1}(0)\cap S^{n-1}_{\epsilon}$ is the \textit{Milnor link}, which $dim~K=n-p-1$.

\begin{example}
If $f:(\mathbb{C}^n,0)\to (\mathbb{C},0)$ is a holomorphic function germ, then it satisfies the Milnor's conditions $(a)$ and $(b)$.

Consider the polynomial map $f:\mathbb{R}^3\to \mathbb{R}^2,~f(x,y,z)=(x,y)$. On has $\nabla f_1=(1,0,0)$ and $\nabla f_2=(0,1,0)$, so there is not critical points. Hence, $f$ satisfies the Milnor's conditions $(a)$ and $(b)$.
\end{example}

Let $p:E\to B$ be a continuous surjection.  A \textit{(homotopy) cross-section} or \textit{section} of $p$ is a (homotopy) right inverse of $p$, i.e., a map $s:B\to E$, such that $p\circ s =1_B$ ($p\circ s \simeq 1_B$). Moreover, given a subspace $A\subset B$, a \textit{(homotopy) local section} of $p$ over $A$ is a (homotopy) section of the restriction map $p_|:p^{-1}(A)\to A$, i.e., a map $s:A\to E$, such that $p\circ s = incl_A$ ($p\circ s \simeq incl_A$).

Since $\mathbb{R}^p$ is contractible, we have that $f$ admits a homotopy continuous cross-section, i.e., a continuous map $s:\mathbb{R}^p\to \mathbb{R}^n$ such that $ps\simeq 1_{\mathbb{R}^p}$. 

On the other hand, since $\text{cat}(S^{p-1}_{\delta})=2$, we have the Scharwz genus $\varrho(f_\mid)\leq 2$. Thus $f_\mid$ admits at most $2$ local continuous sections (see Section \ref{section2}).

Milnor fibration $f_{\mid}$ is \textit{trivial} if $F_f$ is homeomorphic to the unit disk $D^{n-p}$. For example, the projection map germ has trivial Milnor fibration. Thus, trivial Milnor fibration admits a (global) continuous cross-section and so $\varrho(f_{\mid})=1$.

Let $\mathcal{A}$ be a central arrangement in $\mathbb{C}^{d+1}$ (see Section \ref{section-arrangement}). For each hyperplane $H\in\mathcal{A}$, let $L_H:\mathbb{C}^{d+1}\to\mathbb{C}$ be a linear form with kernel $H$. The product $$Q:=Q(\mathcal{A})=\prod_{H\in\mathcal{A}}L_H,$$ then, is a homogeneous polynomial of degree equal to $\mid\mathcal{A}\mid$, the cardinality of the set $\mathcal{A}$.

The complement of the arrangement $$M:=M(\mathcal{A})=\mathbb{C}^{d+1}-\bigcup_{H\in\mathcal{A}}H$$ is a connected, smooth complex variety. 

The polynomial map \[
    Q:M(\mathcal{A})\to \mathbb{C}^\ast
\] is the projection map of a smooth, locally trivial bundle, known as the \textit{Milnor fibration} of the arrangement. Its fiber \[F:=F(\mathcal{A})=Q^{-1}(1)\] is called the \textit{Milnor fiber} of the arrangement. 

 In this work, we prove, in an simple way, that the Milnor fibrations:
 \begin{itemize}
     \item $f_{\mid}:f^{-1}(S^{1}_{\delta})\cap D^{2n+2}_{\epsilon}\to S^{1}_{\delta}$ where $f:(\mathbb{C}^n,0)\to (\mathbb{C},0)$ is a holomorphic function germ (see Proposition \ref{prop-cross-complex}),
     \item $ f_{\mid}:f^{-1}(S^{p-1}_{\delta})\cap D^{n}_{\epsilon}\to S^{p-1}_{\delta}$ where $f:(\mathbb{R}^n,0)\to (\mathbb{R}^p,0)$, $n>p\geq 2$, is an analytic function germ with an isolated singular point at origin (see Proposition \ref{prop-isolated-singularity}),
     \item $Q:M(\mathcal{A})\to\mathbb{C}^\ast$ where $\mathcal{A}$ is a central arrangement in $\mathbb{C}^{d+1}$ (see Proposition \ref{prop-cross-arrangement}),
 \end{itemize} 
  admit continuous (global) cross-sections. Furthermore, we use our results to study the tasking planning problem for the Milnor fibration as a work map and we give the tasking algorithms (see Section \ref{tasking}).

\begin{remark}
The author doesn't know if in general the Milnor fibration $f_{\mid}:f^{-1}(S^{p-1}_{\delta})\cap D^{n}_{\epsilon}\longrightarrow S^{p-1}_{\delta}$ admits a continuous cross-section, for any  $f:(\mathbb{R}^n,0)\longrightarrow (\mathbb{R}^p,0)$, $n> p\geq 2$, analytic map germ which satisfies the Milnor's conditions $(a)$ and $(b)$.
\end{remark}

\section{\bf The Schwarz genus of a map}\label{section2}

In this section we will recall the notion of Schwarz genus and basic results about this numerical invariant (see \cite{schwarz1958genus}). 

Let $p:E\to B$ be a continuous surjection.  A \textit{(homotopy) cross-section} or \textit{section} of $p$ is a (homotopy) right inverse of $p$, i.e., a map $s:B\to E$, such that $p\circ s = 1_B$ ($p\circ s \simeq 1_B$). Moreover, given a subspace $A\subset B$, a \textit{(homotopy) local section} of $p$ over $A$ is a (homotopy) section of the restriction map $p_|:p^{-1}(A)\to A$, i.e., a map $s:A\to E$, such that $p\circ s = incl_A$ ($p\circ s \simeq incl_A$).

\begin{definition}
The \textit{Schwarz genus} or \textit{sectional category} of a surjective continuous map $p:E\to B$ is the smallest number $n$ such that $B$ can be covered by $n$ open subsets $U_1,\ldots,U_n$ with the property that for every $1\leq i\leq n$ there exists a homotopy local continuous cross-section $s_i:U_i\longrightarrow E$, i.e., $p\circ s_i$ is homotopic to the inclusion map $incl_{U_i}:U_i\to B$. The Schwarz genus of a map $p:E\to B$  will be denoted by $\varrho (p)$.
\end{definition}

We recall, that the \textit{standard sectional number} is the minimal number of elements in an open cover of $B$, such that each element admits  a  continuous  local  section  of $p$. Let  us  denote  this quantity  as $\text{sec}_{\text{op}}(p)$.

The following statement is well-known, however, we present a proof. This of course suits best our implementation-oriented objectives.

\begin{lemma}\label{pullback}
Let $p:E\to B$ be a continuous surjection. If the following square
\begin{eqnarray*}
\xymatrix{ E^\prime \ar[r]^{\,\,} \ar[d]_{p^\prime} & E \ar[d]^{p} & \\
       B^\prime  \ar[r]_{\,\, f} &  B &}
\end{eqnarray*}
is a pullback. Then $\text{sec}_{\text{op}}(p^\prime)\leq \text{sec}_{\text{op}}(p)$.
\end{lemma}
\begin{proof}
 Consider the following diagram
\begin{eqnarray*}
\xymatrix{ B^\prime\times_B E \ar@/^10pt/[drr]^{\,\,q_2} \ar@/_10pt/[ddr]_{q_1} \ar@{-->}[dr]_{H} &   &  &\\
& E^\prime \ar[r]^{\,\,} \ar[d]_{p^\prime} & E \ar[d]^{p} & \\
       & B^\prime  \ar[r]_{\,\, f} &  B &}
\end{eqnarray*} where $B^\prime\times_B E=\{(b^\prime,e)\in B^\prime\times E\mid~f(b^\prime)=p(e)\}$ is the \textit{standard pullback}. The map $q_i$ denotes the projection onto its $i-$th coordinate. Let $H:B^\prime\times_B E\to E^\prime$ be a continuous map which makes the triangles commutative, in particular, $p^\prime H=q_1$.

Suppose $U\subset B$ admits a continuous local section of $p$, say $\alpha:U\to E$. Set $V=f^{-1}(U)\subset B^\prime$ and  defines the continuous map $\beta:V\to B^\prime\times_B E$ given by $\beta(b^\prime)=(b^\prime,\alpha(f(b^\prime)))$. It is a local section of $q_1$ over $V$. Then $H\circ \beta$ is a continuous local section of $p^\prime$ over $V$. Thus $\text{sec}_{\text{op}}(p^\prime)\leq \text{sec}_{\text{op}}(p)$.
\end{proof}

\begin{example}
Let $p:E\longrightarrow B$ be a surjective map. If $B$ is contractible, then the Schwarz genus $\varrho(p)=1$. In fact, let $H:B\times [0,1]\longrightarrow B$ be a homotopy with $H_0=1_Y$ and $H_1$ is the constant map $y_0$, where $y_0\in B $. We take $x_0\in E$ such that $p(x_0)=y_0$ and we consider $s:B\to E$ given by $s(y):=x_0$. We note that $1_B\stackrel{H}{\simeq}p\circ s$, thus $p$ admits a homotopy continuous cross-section.
\end{example} 

It is easy to see that for any continuous surjective map $p:E\to B$, it satisfies $\varrho (p)\leq \text{sec}_{\text{op}}(p)$.

\begin{remark}[Sectional category of a fibration]
For a fibration $p:E\longrightarrow B$ the Schwarz genus coincides with the sectional number $sec_{\text{op}}(p)$, i.e., it is the smallest number $n$ such that $B$ can be covered by $n$ open subsets $U_1,\ldots,U_n$ with the property that for each $1\leq i\leq n$ there exists a local continuous cross-section $s_i:U_i\longrightarrow X$ (i.e., $p\circ s_i=incl_{U_i}$).
\end{remark}

The \textit{Lusternik-Schnirelmann category} (LS category) or category of a topological space $X$, denoted cat$(X)$, is the least integer $n$ such that $X$ can be covered with $n$ open sets, say $U_1,\ldots,U_n$, which are all contractible within $X$. Such covering $U_1,\ldots,U_n$ is called \textit{categorical}. We have $\text{cat}(X)=1$ iff $X$ is contractible. The LS category is a homotopy invariant, i.e., if $X$ is homotopy equivalent to $Y$ (we shall denote $X\simeq Y$) then $\text{cat}(X)=\text{cat}(Y)$ (see \cite{cornea2003lusternik} for more properties).

The Schwarz genus satisfies the following basic properties.
\begin{lemma}\emph{(\cite{cornea2003lusternik}, Proposition 9.14)}
Let $p:E\longrightarrow B$ be a fibration. Then
\begin{enumerate}
    \item $\varrho(p)\leq cat(B)$.
    \item If $E$ is contractible, then $\varrho(p)= cat(B)$.
    \item If there are $x_1,\ldots,x_k\in H^\ast(B;R)$ (any coefficient ring $R$) with \[p^\ast(x_1)=\cdots=p^\ast(x_k)=0 \text{ and } x_1\cup\cdots\cup x_k\neq 0,\] \noindent then \[\varrho(p)\geq k+1.\]
\end{enumerate}
\end{lemma}

\begin{remark}
The converse of the item $(2)$ is not true in general, for example the Hopf fibration $S^1\hookrightarrow S^3\longrightarrow S^2$ does not admit a continuous cross-section and thus $\varrho(p)=2=\text{cat}(S^2)$.
\end{remark}

\begin{corollary}
Let $p:E\to B$ be a fibration. If $\varrho(p)=1$ then $p^\ast:H^\ast(B;R)\to H^\ast(E;R)$ is injective.
\end{corollary}

\begin{example}
Let $p:E\to S^m$ be a fibration. If $\varrho(p)=1$ then $p^\ast:H^m(S^m;R)\longrightarrow H^m(E;R)$ is injective.
\end{example}

\begin{lemma}
Let $p:E\to B$ be a fibration. If $\varrho(p)=1$ then $p_\ast:H_\ast(E;R)\to H_\ast(B;R)$ or $p_\#:\pi_\ast(E)\to \pi_\ast(B)$ are surjective.
\end{lemma}
\begin{proof}
If $p$ has a homotopy continuous cross-section $s:B\to E$, then $p\circ s\simeq1_B$ and $p_\ast\circ s_\ast=1_{H_\ast(B;R)}$ and $p_\#\circ s_\#=1_{\pi_\ast(B)}$. Hence, $p_\ast$ and $p_\#$ are surjective.
\end{proof}

\begin{remark}
Let $p:E\to B$ be a fibration with fiber $F$. If $B$ is simply-connected, then the induces map $i_\#:\pi_0(F,e_0)\to \pi_0(E,e_0)$ is a bijection. In particular, $F$ is path-connected if and only if $E$ is path-connected.
\end{remark}

\subsubsection{Fibrations with Schwarz genus one}

Using the  homotopy exact sequence to the fibration $F\stackrel{i}{\hookrightarrow} E\stackrel{p}{\to} B$,
\[\cdots\to\pi_{m+1}(B)\stackrel{\delta_\#}{\to}\pi_{m}(F)\stackrel{i_\#}{\to} \pi_m(E)\stackrel{p_\#}{\to} \pi_m(B)\stackrel{\delta_\#}{\to}\pi_{m-1}(F)\to \cdots\] one has the following statement.
\begin{lemma}\label{fibration-with-corsssection}
If $p$ admits a continuous cross-section $s:(B,b_0)\to (E,e_0)$, then
\begin{itemize}
    \item[(a)] $\pi_q(E,e_0)\cong \pi_q(F\times B,(e_0,b_0))$ for any $q\geq 2$.
    \item[(b)] $\pi_1(E,e_0)\cong \pi_1(F,e_0)\rtimes_\varphi\pi_1(B,b_0)$, where $\varphi:\pi_1(B,b_0)\to \text{Aut}(\pi_1(F,e_0))$ is given by \[(\varphi[\alpha])[\beta]=i_\#^{-1}\left(s_\#([\alpha])i_\#([\beta])s_\#([\alpha]^{-1})\right) \text{ for any } [\alpha]\in \pi_1(B,b_0) \text{ and } [\beta]\in \pi_1(F,e_0).\]
\end{itemize}
\end{lemma}

Now, from the homotopy exact sequence to the fibration $F\stackrel{i}{\hookrightarrow} E\stackrel{p}{\to} S^m$,
\[\cdots\to \pi_m(E)\stackrel{p_\#}{\to} \pi_m(S^m)\stackrel{\delta_\#}{\to}\pi_{m-1}(F)\stackrel{i_\#}{\to}\pi_{m-1}(E)\to 1\] one has the following statement.
\begin{lemma}\label{characterization-cross-over-sphere}
Let $p:E\longrightarrow S^m$ be a fibration with fiber $F$. Then the following statements are equivalent:
\begin{itemize}
    \item[(i)] $p$ admits a continuous cross-section $s:(S^m,1)\to (E,e_0)$.
    \item[(ii)] $p_\#:\pi_m(E,e_0)\to \pi_m(S^m,1)$ is surjective.
    \item[(iii)] $\delta_\#:\pi_m(S^m,1)\to \pi_{m-1}(F,e_0)$ is trivial.
    \item[(iv)] $i_\#:\pi_{m-1}(F,e_0)\to \pi_{m-1}(E,e_0)$ is a bijection.
\end{itemize}
\end{lemma}

\begin{corollary}
Let $p:E\to S^1$ be a fibration with fiber $F$. If $p$ admits a continuous cross-section $s:(S^1,1)\to (E,e_0)$, then $i_\#:\pi_{0}(F,e_0)\to \pi_{0}(E,e_0)$ is a bijection. In particular, $F$ is path-conected if and only if $E$ is path connected.
\end{corollary}

\begin{example}\label{over-s1}
Any fibration $p:E\to S^1$ admits a continuous cross-section if its fiber $F$ is path-connected. 
\end{example}

\section{\bf Milnor's fibrations}

\subsection{Milnor's tube fibration}
In this section we will recall the Milnor's tube fibration from (\cite{dutertre2016}, Corollary 2.5). 

We shall $D^n_\epsilon$ denotes the closed ball in $\mathbb{R}^n$ with radius $\epsilon$ and $S^n_\epsilon$ denotes the sphere in $\mathbb{R}^n$ with radius $\epsilon$. 

Let $f:(\mathbb{R}^n,0)\to (\mathbb{R}^p,0)$, $n> p\geq 2$, be an analytic map germ which satisfies the Milnor's conditions $(a)$ and $(b)$, $V=f^{-1}(0)$. By (\cite{dutertre2016}, Corollary 2.5), one has the Milnor's fibration.

\begin{proposition}
Let $f:(\mathbb{R}^n,0)\to (\mathbb{R}^p,0)$ as above and $\epsilon_0>0$ be a Milnor's radius for $f$ at origin. Then, for each $0<\epsilon\leq \epsilon_0$, there exists $\delta$, $0<\delta\ll\epsilon$, such that 
\begin{equation}\label{milnor-fibration-section}
   f_{\mid}:f^{-1}(S^{p-1}_{\delta})\cap D^{n}_{\epsilon}\to S^{p-1}_{\delta},  
\end{equation} is a projection of a smooth locally trivial fiber bundle. In particular, it is a fibration.
\end{proposition}

Denote the \textit{Milnor tube} by $M(\delta,\epsilon)=f^{-1}(S^{p-1}_{\delta})\cap D^{n}_{\epsilon}$, and note that it has $codim(M(\delta,\epsilon))=codim(f^{-1}(S^{p-1}_{\delta}))+codim(D^{n}_{\epsilon})=codim(S^{p-1}_{\delta})+0=1$, so $dim(M(\delta,\epsilon))=n-1$.

Its fiber $f^{-1}(\delta)$ is called \textit{the Milnor fiber} denoted by $F_f$, which is a compact manifold with boundary. We have $dim~F_f=n-p$ and $\partial F_f$ is diffeomorphic to $K$, where $K:=f^{-1}(0)\cap S^{n-1}_{\epsilon}$ is the \textit{Milnor link}, which $dim~K=n-p-1$.

\begin{example}(Complex case)\label{exam-complex}
Let $f:(\mathbb{C}^n,0)\to (\mathbb{C},0)$ be a holomorphic function germ, then it satisfies the Milnor conditions $(a)$ and $(b)$ (see \cite{dutertre2016}, Corollary 2.5). Thus its Milnor fibration is
\begin{equation}\label{milnor-fibration-complex}
   f_{\mid}:f^{-1}(S^{1}_{\delta})\cap D^{2n+2}_{\epsilon}\to S^{1}_{\delta}.  
\end{equation} Furthermore, if $f$ has an isolated singularity at origin, then the Milnor fiber $F_f$ is $(n-1)-$connected and has the homotopy type of a wedge of spheres $S^n\vee\cdots\vee S^n$ (with $\mu_f$ spheres). $\mu_f$ is called the Milnor number. More generally, any holomorphic function germ  $f:(\mathbb{C}^n,0)\to (\mathbb{C},0)$ has Milnor fiber $(n-s-2)-$connected, where $s=\text{dim}_{\mathbb{C}} \sum_f$. 
\end{example}

\begin{example}(Isolated singularity)\label{isolated-singular}
Let $f:(\mathbb{R}^n,0)\to (\mathbb{R}^p,0)$, $n>p\geq 2$, be an analytic function germ with an isolated singular point at origin. Then $f$ satisfies the Milnor conditions $(a)$ and $(b)$, and thus its Milnor fibration is
\begin{equation}\label{milnor-fibration-isolated}
   f_{\mid}:f^{-1}(S^{p-1}_{\delta})\cap D^{n}_{\epsilon}\to S^{p-1}_{\delta}.  
\end{equation} Furthermore, by Milnor \cite{milnor1968}, the Milnor fiber $F_f$ is $(p-2)-$connected.
\end{example}

\subsection{Milnor's fibration for arrangements}\label{section-arrangement}
In this section we recall the definitions and properties of Milnor's fibration of arrangements given by Suciu in \cite{suciu2017}.

An \textit{arrangement of hyperplanes} is a finite set $\mathcal{A}$ of codimension 1 linear
subspaces in a finite-dimensional, complex vector space $\mathbb{C}^{d+1}$. 

We will assume throughout that the arrangement is \textit{central}, that is, all the hyperplanes pass through the origin. For each hyperplane $H\in\mathcal{A}$, let $L_H:\mathbb{C}^{d+1}\to\mathbb{C}$ be a linear form with kernel $H$. 

The product $$Q:=Q(\mathcal{A})=\prod_{H\in\mathcal{H}}L_H,$$ then, is a homogeneous polynomial of degree equal to $\mid\mathcal{A}\mid$, the cardinality of the set $\mathcal{A}$. Note that $Q$ is a defining polynomial for the arrangement, unique up to a (non-zero) constant factor.

The complement of the arrangement $$M:=M(\mathcal{A})=\mathbb{C}^{d+1}-\bigcup_{H\in\mathcal{H}}H$$ is a connected, smooth complex variety. Moreover, $M(\mathcal{A})$ has the homotopy type of a finite CW-complex of dimension at most $d+1$. 

The polynomial map \begin{equation}\label{milnor-fibration-arrangement}
    Q:M(\mathcal{A})\to \mathbb{C}^\ast
\end{equation} is the projection map of a smooth, locally trivial bundle, known as the \textit{Milnor fibration} of the arrangement. Its fiber \[F:=F(\mathcal{A})=Q^{-1}(1)\] is called the \textit{Milnor fiber} of the arrangement. $F$ is a smooth, connected, orientable manifold of dimension $2d$. Moreover, $F$ has the homotopy type of a finite CW-complex of dimension $d$. If $Q$ has an isolated singularity at origin (for example, if $d=1$), then $F$ is homotopic to a wedge of $d-$spheres, and so $F$ is $(d-1)-$connected.

\section{\bf Cross-sections for Milnor's fibrations}\label{sect-cross-section}

In this section we will study the existence of cross-sections for the Milnor's fibrations given in the previous section.

\begin{proposition}\label{prop-cross-complex}
Let $f:(\mathbb{C}^n,0)\to (\mathbb{C},0)$ be a holomorphic function germ, then the Milnor fibration (\ref{milnor-fibration-complex}) admits a continuous cross-section.
\end{proposition}
\begin{proof}
It follows from Example \ref{over-s1} and Example \ref{exam-complex}.
\end{proof}

By Lemma \ref{fibration-with-corsssection}, one has the following corollary.

\begin{corollary}
Let $f:(\mathbb{C}^n,0)\to (\mathbb{C},0)$ be a holomorphic function germ, then
\begin{itemize}
    \item[(a)] $\pi_q\left(f^{-1}(S^{1}_{\delta})\cap D^{2n+2}_{\epsilon}\right)\cong \pi_q(F_f)$ for any $q\geq 2$.
    \item[(b)] $\pi_1\left(f^{-1}(S^{1}_{\delta})\cap D^{2n+2}_{\epsilon}\right)\cong \pi_1(F_f)\rtimes_\varphi\mathbb{Z}$, where $\varphi:\mathbb{Z}\to \text{Aut}(\pi_1(F_f))$ is given by \[(\varphi[\alpha])[\beta]=i_\#^{-1}\left(s_\#([\alpha])i_\#([\beta])s_\#([\alpha]^{-1})\right) \text{ for any } [\alpha]\in \pi_1(B)\cong\mathbb{Z} \text{ and } [\beta]\in \pi_1(F_f).\]
\end{itemize}
\end{corollary}

\begin{proposition}\label{prop-isolated-singularity}
Let $f:(\mathbb{R}^n,0)\to (\mathbb{R}^p,0)$, $n>p\geq 2$, be an analytic function germ with an isolated singular point at origin. Then the Milnor fibration \ref{milnor-fibration-isolated} admits a continuous cross-section.
\end{proposition}
\begin{proof}
It follows from Lemma \ref{characterization-cross-over-sphere} and Example \ref{isolated-singular}.
\end{proof}

Once again, by Lemma \ref{fibration-with-corsssection}, one has the following corollary.

\begin{corollary}
Let $f:(\mathbb{R}^n,0)\to (\mathbb{R}^p,0)$, $n>p\geq 2$, be an analytic function germ with an isolated singular point at origin. Then
\begin{itemize}
    \item[(a)] $\pi_q\left(f^{-1}(S^{p-1}_{\delta})\cap D^{n}_{\epsilon}\right)\cong \pi_q(F_f\times S^{p-1}_{\delta})$ for any $q\geq 2$.
    \item[(b)] $\pi_1\left(f^{-1}(S^{p-1}_{\delta})\cap D^{n}_{\epsilon}\right)\cong \pi_1(F_f)\rtimes_\varphi \pi_1(S^{p-1})$, where $\varphi:\pi_1(S^{p-1}_\delta)\to \text{Aut}(\pi_1(F_f))$ is given by \[(\varphi[\alpha])[\beta]=i_\#^{-1}\left(s_\#([\alpha])i_\#([\beta])s_\#([\alpha]^{-1})\right) \text{ for any } [\alpha]\in \pi_1(S^{p-1}_\delta) \text{ and } [\beta]\in \pi_1(F_f).\]
\end{itemize}
\end{corollary}

\begin{proposition}\label{prop-cross-arrangement}
Let $\mathcal{A}$ be a central arrangement in $\mathbb{C}^{d+1}$, then the Milnor fibration \ref{milnor-fibration-arrangement} admits a continuous cross-section.
\end{proposition}
\begin{proof}
It follows from Example \ref{over-s1} and Section \ref{section-arrangement}.
\end{proof}

Once again, by Lemma \ref{fibration-with-corsssection}, one has the following corollary.

\begin{corollary}
Let $\mathcal{A}$ be a central arrangement in $\mathbb{C}^{d+1}$, then
\begin{itemize}
    \item[(a)] $\pi_q\left(M(\mathcal{A})\right)\cong \pi_q(F(\mathcal{A}))$ for any $q\geq 2$.
    \item[(b)] $\pi_1\left(M(\mathcal{A})\right)\cong \pi_1(F(\mathcal{A}))\rtimes_\varphi\mathbb{Z}$, where $\varphi:\mathbb{Z}\to \text{Aut}\left(\pi_1(F(\mathcal{A}))\right)$ is given by \[(\varphi[\alpha])[\beta]=i_\#^{-1}\left(s_\#([\alpha])i_\#([\beta])s_\#([\alpha]^{-1})\right) \text{ for any } [\alpha]\in \pi_1(\mathbb{C}^\ast)\cong\mathbb{Z} \text{ and } [\beta]\in \pi_1(F(\mathcal{A})).\]
\end{itemize}
\end{corollary}

\section{\bf Tasking planning}\label{tasking}

In robotics (see \cite{bajdRobotics}), recall that the \textit{configuration space} or \textit{state space} of a system $\mathcal{S}$ is defined as the space of all possible states of $\mathcal{S}$. The \textit{robot workspace} consists of all points that can be reached by the robot end-point, i.e., the space of all tasks. The configuration space and workspace are often described as a subspace of some Euclidean space $\mathbb{R}^n$ and $\mathbb{R}^p$ respectively. The \textit{work map} or \textit{kinematic map} is a continuous map from the configuration $C$ to the workspace $W$, that is, it is a continuous map \[f:C\to W\] which assigns to each state of the configuration space the position of the end-effector at that state. This map is an important object to be considered when implementing algorithms controlling the task performed by the robot manipulator. 

For our purposes, in this section we consider $n> p\geq 2$ and $f:(\mathbb{R}^n,0)\longrightarrow (\mathbb{R}^p,0)$ be an analytic map germ which satisfies the Milnor's conditions $(a)$ and $(b)$. In particular, we have the Milnor fibration as work map: \begin{equation}\label{milnor-fibration}
   f_{\mid}:f^{-1}(S^{p-1}_{\delta})\cap D^{n}_{\epsilon}\to S^{p-1}_{\delta},  
\end{equation}
where $0<\delta\ll\epsilon\leq \epsilon_0,$ and $\epsilon_0$ is a Milnor's radius for $f$ at origin. Recall that, $M(\delta,\epsilon)=f^{-1}(S^{p-1}_{\delta})\cap D^{n}_{\epsilon}$ denote the Milnor tube.

Our work in this section is considered as an instance of planning. \textit{The planning problem} consists in implementing algorithms (we will say tasking planning algorithm) controlling the task performed by the robot manipulator. The challenges of modern robotics (see, for example Latombe \cite{latombe2012robot} and LaValle \cite{lavalle2006planning}) is design explicit and suitably optimal motion planners. The input of such an algorithm (see P. Pavesic \cite{pavesic}) are pairs $(a,A)\in M(\delta,\epsilon)\times S^{p-1}_{\delta}$ of points of the configuration space and workspace, respectively. The output for such a pair ought to be a path in the configuration space $\alpha\in PM(\delta,\epsilon)$ such that $\alpha(0)=a$ and $f(\alpha(1))=A$. To give algorithms we need to know the smallest possible number of regions of continuity for any planning algorithm, that is, the value of the \textit{topological complexity} a la Pavesic \cite{pavesic} of the work map $f_{\mid}$,  $\text{TC}(f_{\mid})$.

Recall that $PE$ denote the space of all continuous paths $\gamma: [0,1] \longrightarrow E$ in $E$ and  $e_{0,1}: PE \longrightarrow E \times E$ denote the fibration associating to any path $\gamma\in PE$ the pair of its initial and end points $\pi(\gamma)=(\gamma(0),\gamma(1))$. Equip the path space $PE$ with the compact-open topology. Let $p:E\to B$ be a continuous surjection between path-connected spaces, and let \begin{equation}
    e_p:PE\to E\times B,~e_p=(1\times p)\circ e_{0,1}. 
\end{equation}

A \textit{tasking planning  algorithm} is  a  cross-section $s\colon E\times B\to PE$ of  the  fibration  $e_p$, i.e.~ a (not necessarily continuous) map satisfying $e_p\circ s=1_{E\times B}$.

Now, we recall that the \textit{topological complexity} of the map $p$, denoted by TC$(p)$, is the sectional number $sec_{op}(e_p)$ of the map $e_p$, that is, is the least integer $m$ such that the Cartesian product $E\times B$ can be covered with $m$ open subsets $U_i$, \begin{equation*}
        E \times B = U_1 \cup U_2 \cup\cdots \cup U_m, 
    \end{equation*} such that for any $i = 1, 2, \ldots , m$ there exists a continuous local section $s_i : U_i \longrightarrow PE$ of $e_{P}$, that is, $e_{P}\circ s_i = incl_{U_i}$. If no such $m$ exists we will set TC$(p)=\infty$. 
    
Here, note that that the topological complexity of the identity map $1_E:E\to E$, $\text{TC}(id_E)=\text{TC}(E)$, coincides with the topological complexity (a la Farber) of $E$ (see M. Farber \cite{farber2003topological}). We have $\text{TC}(X)=1$ iff $X$ is contractible. The TC is also a homotopy invariant, i.e., if $X\simeq Y$ then $\text{TC}(X)=\text{TC}(Y)$. Moreover, in general $\text{cat}(X)\leq \text{TC}(X)$. 

Any tasking planning algorithm $s:=\{s_i:U_i\to PE\}_{i=1}^{n}$ is called \textit{optimal} if $n=\text{TC}(p)$.

\begin{remark}\label{fibration-implies-pullback}
Suppose $p:E\to B$ is a fibration and $p^\prime:B\to B^\prime$ is a map. Then for any $\beta:X\to PB$ and any $\alpha:X\to E\times B^\prime$ satisfying $e_{p^\prime}\circ\beta=(p\times 1_{B^\prime})\circ\alpha$, there exists $H:X\to PE$ such that
$e_{p^\prime\circ p}\circ H=\alpha$ and $p_{\#}\circ H=\beta$.
\begin{eqnarray*}
\xymatrix{ X \ar@/^10pt/[drr]^{\,\,\beta} \ar@/_10pt/[ddr]_{\alpha} \ar@{-->}[dr]_{H} &   &  &\\
& PE \ar[r]^{\,\,p_{\#}} \ar[d]^{e_{p^\prime\circ p}} & PB \ar[d]^{e_{p^\prime}} & \\
       & E\times B^\prime   \ar[r]_{\quad p\times 1_{B^\prime}\quad} &  B\times B^{\prime} &}
\end{eqnarray*} Note that we have the following commutative diagram:
\begin{eqnarray*}
\xymatrix{ X \ar[r]^{\,\,p_{1}\circ\alpha} \ar[d]_{i_0} & E \ar[d]^{p} \\
       X\times I  \ar[r]_{\,\,\beta} &  B}
\end{eqnarray*} where $p_1$ is the projection onto the first coordinate. Because $p$ is a fibration, there exists $H:X\times I\to E$ satisfying $H\circ i_0=p_1\circ\alpha$ and $p\circ H=\beta$, thus we does.
\end{remark}

The following statement was presented by Pavesic in \cite{pavesic2019}. However, we present a proof which of course suits best our implementation-oriented objectives.

\begin{lemma}\label{key-lem}
If $p:E\to B$ is a fibration that admits a continuous cross-section, then \[\text{TC}(p)=\text{TC}(B).\]
\end{lemma}
\begin{proof}
Consider the diagram of maps $E\stackrel{p}{\to} B\stackrel{p^{\prime}}{\to}B^\prime$ and let $s:B\to E$ be a cross-section to $p$. We will show that $\text{TC}(p^{\prime})\leq \text{TC}(p^{\prime} p)$. In particular, $\text{TC}(B)\leq\text{TC}(p)$. We will check $\text{TC}(p^{\prime})\leq \text{TC}(p^{\prime} p)$. Suppose $\alpha_{p^{\prime} p}:U\to PE$ is a local section of $e_{p^{\prime} p}$ over $U\subset E\times B^\prime$. Set $V:=(s\times 1_{B^\prime})^{-1}(U)\subset B\times B^\prime$. Then we can define the continuous map $\alpha_{p^{\prime}}:V\to PB$ by \[\alpha_{p^{\prime}}(b,b^\prime)(t):=\begin{cases}
    b, & \hbox{for $0\leq t\leq \frac{1}{2}$;} \\
    p(\alpha_{p^{\prime} p}(s(b),b^\prime)(2t-1)), & \hbox{for $\frac{1}{2}\leq t\leq 1$.}
\end{cases}\] We have $\alpha_{p^{\prime}}$ is a local section of $e_{p^{\prime}}$ over $V$. Thus, we conclude that $\text{TC}(p^{\prime})\leq \text{TC}(p^{\prime} p)$.

Next, we will show that $\text{TC}(p^{\prime} p)\leq \text{TC}(p^{\prime})$. In particular, $\text{TC}(p)\leq \text{TC}(B)$. We will check $\text{TC}(p^{\prime} p)\leq \text{TC}(p^{\prime})$. Since $p:E\to B$ is a fibration, the following diagram is a pullback (see Remark \ref{fibration-implies-pullback}) \begin{eqnarray*}
\xymatrix{ PE \ar[r]^{\,\,p_{\#}} \ar[d]_{e_{p^\prime p}} & PB \ar[d]^{e_{p^\prime}} & \\
       E\times B^\prime  \ar[r]_{\,\, p\times 1_{B^\prime}} &  B\times B^\prime &}
\end{eqnarray*} This implies $\text{TC}(p^\prime p)=\text{sec}_{\text{op}}(e_{p^\prime p})\leq \text{sec}_{\text{op}}(e_{p^\prime})= \text{TC}(p^\prime)$. 

Therefore, $\text{TC}(p)=\text{TC}(B)$, establishing the result. 
\end{proof}

\begin{example}\label{spheres}[For spheres]
By M. Farber \cite{farber2003topological}, the TC for spheres is as follows:
\[ \text{TC}(S^m)=
\begin{cases}
2, & \hbox{ if $m$ is odd;}\\
3, & \hbox{ if $m$ is even.}
\end{cases} \]

Furthermore, suppose $m$ is odd. Let $v$ denote a fixed continuous unitary tangent vector field on $S^{m}$, say $v(x_1,y_1,\ldots,x_\ell,y_\ell)=(-y_1,x_1,\ldots,-y_\ell,x_\ell)$ with $m+1=2\ell$. 

A planning algorithm to $S^m$ is given by $s:=\{s_i:U_i\to PS^m\}_{i=1}^{2}$, where \begin{eqnarray*}
U_1&=&  \{(\theta_1,\theta_2)\in S^m\times S^m\mid ~~\theta_1\neq -\theta_2\},\\
U_2&=& \{(\theta_1,\theta_2)\in S^m\times S^m\mid ~~\theta_1\neq \theta_2\},
\end{eqnarray*} with local algorithms: 
$$ 
s_1(\theta_1,\theta_2)(t) = \dfrac{(1-t)\theta_1+t\theta_2}{\parallel (1-t)\theta_1+t\theta_2 \parallel} \text{ for all } (\theta_1,\theta_2)\in U_1.
$$ Before to define $s_2$, we can consider the subset $F:=\{(\theta_1,\theta_2)\in S^m\times S^m\mid ~~\theta_1= -\theta_2\} $ and for $(\theta_1,\theta_2)\in F$ define
$$
\alpha(\theta_1,\theta_2)(t) =  \begin{cases} 
s_1(\theta_1,v(\theta_1))(2t), &\hbox{ $0\leq t\leq\dfrac{1}{2}$;}\\
s_1(v(\theta_1),\theta_2)(2t-1), &\hbox{ $\dfrac{1}{2}\leq t\leq 1$.}
\end{cases} $$ Now, for all  $(\theta_1,\theta_2)\in U_2$, set
$$ 
s_2(\theta_1,\theta_2)(t) =  \begin{cases} 
s_1(\theta_1,-\theta_2)(2t), &\hbox{ $0\leq t\leq\dfrac{1}{2}$;}\\
\alpha(-\theta_2,\theta_2)(2t-1), &\hbox{ $\dfrac{1}{2}\leq t\leq 1$,}
\end{cases} $$

For our purposes, we present an optimal algorithm for even-dimensional spheres. Suppose $m$ is even. Let $\nu:S^m\to \mathbb{R}^{m+1}$ denote a fixed continuous tangent vector field on $S^{m}$, which vanishes at two points $-1,1\in S^m$ and is nonzero for any $x\in S^m-\{-1,1\}$. Indeed, we can take $$\nu(x_1,x_2,x_3,\ldots,x_{m},x_{m+1})=(0,-x_3,x_2,\ldots,-x_{m+1},x_{m}).$$ Note that $\nu(1,0,\ldots,0)=0$, $\nu(-1,0\ldots,0)=0$ and $\nu(x)\neq 0$ for any $x\in S^m-\{-1,1\}$.   

A planning algorithm to $S^m$ with $m$ even is given by $\kappa:=\{\kappa_i:V_i\to PS^m\}_{i=1}^{3}$, where \begin{eqnarray*}
V_1&=&  \{(\theta_1,\theta_2)\in S^m\times S^m\mid ~~\theta_1\neq -\theta_2\},\\
V_2&=& \{(\theta_1,\theta_2)\in S^m\times S^m\mid ~~\theta_1\neq \theta_2 \text{ and } \theta_2\neq -1,1\},
\end{eqnarray*} with local algorithms: 
$$ 
\kappa_1(\theta_1,\theta_2)(t) = \dfrac{(1-t)\theta_1+t\theta_2}{\parallel (1-t)\theta_1+t\theta_2 \parallel} \text{ for all } (\theta_1,\theta_2)\in V_1.
$$ Before to define $\kappa_2$, we can consider the subset $G:=\{(\theta_1,\theta_2)\in S^m\times S^m\mid ~~\theta_1= -\theta_2 \text{ and } \theta_2\neq -1,1\} $ and for $(\theta_1,\theta_2)\in G$ define
$$
\beta(\theta_1,\theta_2)(t) =  \begin{cases} 
\kappa_1(\theta_1,\nu(\theta_1))(2t), &\hbox{ $0\leq t\leq\dfrac{1}{2}$;}\\
\kappa_1(\nu(\theta_1),\theta_2)(2t-1), &\hbox{ $\dfrac{1}{2}\leq t\leq 1$.}
\end{cases} $$ Now, for all  $(\theta_1,\theta_2)\in V_2$, set
$$ 
\kappa_2(\theta_1,\theta_2)(t) =  \begin{cases} 
\kappa_1(\theta_1,-\theta_2)(2t), &\hbox{ $0\leq t\leq\dfrac{1}{2}$;}\\
\beta(-\theta_2,\theta_2)(2t-1), &\hbox{ $\dfrac{1}{2}\leq t\leq 1$,}
\end{cases} $$

Now, $V_1\cup V_2$ covers everything except the points $(-1,1)$ and $(1,-1)$. We recall the stereographic projection with respect to the north pole $p_N=(0,\ldots,0,1)$: \[p:S^m-\{p_N\}\to \mathbb{R}^m,~(x_1,\ldots,x_{m+1})=\left(\dfrac{x_1}{1-x_{m+1}},\ldots,\dfrac{x_m}{1-x_{m+1}}\right)\] whose inverse is given by \[q:\mathbb{R}^m\to S^m-\{p_N\},~y=(y_1,\ldots,y_m)=\left(\dfrac{2y_1}{\parallel y\parallel^2+1},\ldots,\dfrac{2y_m}{\parallel y\parallel^2+1},\dfrac{\parallel y\parallel^2-1}{\parallel y\parallel^2+1} \right).\] Set $Y=S^m-\{p_N\}$. This means that we may take $V_3=Y\times Y$. The algorithm $\kappa_3:V_3\to PS^m$ is given by 
\[\kappa(\theta_1,\theta_2)(t)= q\left((1-t)p(\theta_1)+t(\theta_2)\right) \text{ for } (\theta_1,\theta_2)\in V_3 \text{ and } t\in[0,1].\]
\end{example}

In this section we compute the value of $\text{TC}(f_{\mid})$ when \begin{enumerate}
    \item $f:(\mathbb{C}^n,0)\to (\mathbb{C},0)$ is a holomorphic function germ.
    \item $f:(\mathbb{R}^n,0)\to (\mathbb{R}^p,0)$, $n>p\geq 2$, is an analytic function germ with an isolated singular point at origin.
    \item $f_{\mid}:M(\mathcal{A})\to \mathbb{C}^\ast$ is a Milnor fibration of the central arrangement $\mathcal{A}$.
\end{enumerate}

\begin{theorem}
 We have $$\text{TC}(f_{\mid})=\begin{cases}
 2, & \hbox{ if $f$ is as (1), (3) or (2) with $p$ even;}\\
 3, & \hbox{ if $f$ is as (2) with $p$ odd.}
 \end{cases}$$. 
\end{theorem}
\begin{proof}
By Section \ref{sect-cross-section}, the Milnor fibration $f_{\mid}$ admits a continuous cross-section, then by Lemma \ref{key-lem}, the topological complexity $\text{TC}(f_{\mid})=\text{TC}(S^m)$ and thus the Example \ref{spheres} implies the result.
\end{proof}

\subsection{Tasking planning algorithms}

In this section, we present optimal tasking planning algorithms with $2$ and $3$ regions of continuity, respectively. These algorithms are easily implementable in practice.

Let $f:(\mathbb{R}^n,0)\to (\mathbb{R}^p,0)$, $n>p\geq 2$, be an analytic function germ with an isolated singular point at origin. We present a tasking planning algorithm for the Milnor fibration $f_{\mid}:M(\delta,\epsilon)\to S^{p-1}$. Recall the following pullbacks:
\begin{eqnarray*}
\xymatrix{ B^\prime\times_B E \ar@/^10pt/[drr]^{\,\,q_2} \ar@/_10pt/[ddr]_{q_1} \ar@{-->}[dr]_{H} &   &  &\\
& PM(\delta,\epsilon) \ar[r]^{\,\,(f_{\mid})_{\#}} \ar[d]_{e_{f_{\mid}}} & PS^{p-1} \ar[d]^{e_{0,1}} & \\
       & M(\delta,\epsilon)\times S^{p-1}   \ar[r]_{\quad f_{\mid}\times 1\quad} &  S^{p-1}\times S^{p-1} &}
\end{eqnarray*} where $B^\prime=M(\delta,\epsilon)\times S^{p-1}$, $B=S^{p-1}\times S^{p-1}$ and $E= PS^{p-1}$. Recall that $H$ is given by the following commutative diagram (see Remark \ref{fibration-implies-pullback}):
\begin{eqnarray*}
\xymatrix{ X \ar[r]^{\,\,p_{1}\circ q_1} \ar[d]_{i_0} & M(\delta,\epsilon) \ar[d]^{f_{\mid}} \\
       X\times I\ar@{-->}[ur]_{H}  \ar[r]_{\,\,q_2} &  S^{p-1}}
\end{eqnarray*} where $X=B^\prime\times_B E$.

\subsubsection{For $p$ even} From Lemma \ref{key-lem} and Lemma \ref{pullback}, the optimal algorithm $s:=\{s_i:U_i\to PS^{p-1}\}_{i=1}^{2}$ on the sphere $S^{p-1}$ (see Example \ref{spheres}) induces an optimal tasking algorithm for $f_{\mid}$, say $\hat{s}:=\{\hat{s}_i:\widehat{V}_i\to PM(\delta,\epsilon)\}_{i=1}^{2}$, where $\widehat{V}_i=(f_{\mid}\times 1)^{-1}(U_i)\subset M(\delta,\epsilon)\times S^{p-1}$ and $\hat{s}_i(v)=H(v,s_i\circ(f_{\mid}\times 1)(v))$.

\subsubsection{For $p$ odd} Again, from Lemma \ref{key-lem} and Lemma \ref{pullback}, the optimal algorithm $\kappa:=\{\kappa_i:V_i\to PS^{p-1}\}_{i=1}^{3}$ on the sphere $S^{p-1}$ (see Example \ref{spheres}) induces an optimal tasking algorithm for $f_{\mid}$, say $\hat{\kappa}:=\{\hat{\kappa}_i:\widehat{V}_i\to PM(\delta,\epsilon)\}_{i=1}^{3}$, where $\widehat{V}_i=(f_{\mid}\times 1)^{-1}(V_i)\subset M(\delta,\epsilon)\times S^{p-1}$ and $\hat{s}_i(v)=H(v,\kappa_i\circ(f_{\mid}\times 1)(v))$.

\bibliographystyle{plain}

\end{document}